\documentclass[a4paper]{amsart}
\usepackage{amsfonts,amsmath,amssymb,amscd,amstext,amsbsy,latexsym}
\newtheorem{thm}{Theorem}[section]
\newtheorem{cor}[thm]{Corollary}
\newtheorem{lem}[thm]{Lemma}
\newtheorem{prop}[thm]{Proposition}
\theoremstyle{definition}

\theoremstyle{remark}
\newtheorem{rem}[thm]{Remark}

\numberwithin{equation}{section}

\begin{document}

\title{Maximal subrings of Division Rings}%
\author{Alborz Azarang}%

\maketitle

\centerline{Department of Mathematics, Faculty of Mathematical Sciences and Computer,}
\centerline{ Shahid Chamran University
of Ahvaz, Ahvaz-Iran} \centerline{a${}_{-}$azarang@scu.ac.ir}
\centerline{ORCID ID: orcid.org/0000-0001-9598-2411}
\begin{abstract}
The structure and the existence of maximal subrings in division rings are investigated. We see that if $R$ is a maximal subring of a division ring $D$ with center $F$ and $N(R)\neq U(R)\cup \{0\}$, where $N(R)$ is the normalizer of $R$ in $D$, then either $R$ is a division ring with $[D:R]_l=[D:R]_r$ is finite or $R$ is an Ore $G$-domain with certain properties. In particular, if $F\subsetneq C_D(R)$, the centralizer of $R$ in $D$, then $R=C_D(\beta)$ is a division ring, for each $\beta\in C_R(R)\setminus F$, $[D:R]_l$ is finite if and only if $\beta$ is algebraic over $F$, $[D:R]_l=[D:R]_r=[F[\beta]:F]$ and $C_R(R)=F[\beta]$. On the other hand if $R$ does not contains $F$, then $R\cap F=C_R(R)$ is a maximal subring of $F$. Consequently, if a division ring $D$ has a noncentral element which is algebraic over the center of $D$, then $D$ has a maximal subring. In particular, we prove that  if $D$ is a non-commutative division ring with center $F$, then either $D$ has a maximal subring or $dim_F(D)\geq |F|$. We study when a maximal subring of a division ring is a left duo ring or certain valuation rings. Finally, we prove that if $D$ is an existentially complete division ring over a field $K$, then $D$ has a maximal subring of the form $C_D(x)$ where $D$ is finite over it. Moreover, if $R$ is a maximal subring of $D$ with $K\subsetneq C_R(R)$, then $R=C_D(x)$ for some $x\in D\setminus K$, which is algebraic over $K$.
\end{abstract}
\section{Introduction}
If $K$ is a field, then the structure of maximal subrings of $K$ are easy to characterize. In fact, it is not hard to see that a proper subring $R$ of $K$ is a maximal subring of $K$ if and only if either $R$ is a field and $[K:R]$ is finite of minimum degree or $R$ is a one-dimensional valuation for $K$ (in particular in this case $R$ is a $G$-domain with quotient field $K$). On the other hand, the existence of such subrings for fields exactly determined in \cite{azkrf}. In fact, most of fields have maximal subrings and in particular, a field without maximal subrings are certain absolutely algebraic fields (i.e., a certain field with nonzero characteristic which is algebraic over its prime subfield). More exactly a field $K$ has no maximal subrings if and only if there exists a prime number $p$ such that

$$K=\bigcup_{n|S}\mathbb{F}_{p^n}$$

where $\mathbb{F}_{p^n}$ is a field with $p^n$ elements and $S$ is a supper natural number (or a Steinitz's number) of the form $\prod_{i=1}^\infty q_i^{n_i}$, $\mathbb{P}=\{q_1=2, q_2=3, \ldots\}$ is the set of all prime numbers and $n_i$ is either $0$ or $\infty$, see \cite[Defenition 1.5]{azkrf} and \cite[Remark 2]{azffs}. Moreover, if $\mathcal{F}$ is the set of all fields, up to isomorphisms, without maximal subrings then $|\mathcal{F}|=2^{\aleph_0}$, see \cite[Corollary 1.15]{azkrf}. We refer the reader to \cite{azpa} for some facts about maximal subrings up to isomorphism of fields too. Also it is interesting to know that if $T$ is a commutative ring and $U(T)$ is not integral over the prime subring of $T$, then $T$ has a maximal subring, see \cite[Theorem 2.4]{azconch}.\\

In this paper motivated by the results for fields, we are interested to investigate the structure of maximal subrings and the existence of them for division rings. As we see in this paper the existence of maximal subrings of division rings is different from fields and depend to the center of division rings and noncentral elements, but the structure of maximal subrings in division rings is similar to fields. For example, if $D$ is a non-commutative division ring without maximal subrings, then each noncentral element of $D$ is not algebraic over the center of $D$ (and therefore is not algebraic over the prime subring of $D$), but if $R$ is a maximal subring of $D$, then in more situations either $R$ is a division ring with $[D:R]$ is finite with minimum degree or $R$ is an Ore $G$-domain with division ring of quotients $D$, and has certain algebraic properties similar to non-field maximal subrings of a field.\\

A brief outline of this paper is as follows. In the Section 2, we see that if $R$ is a maximal subring of a division ring $D$ with center $F$, then one of the following holds:
\begin{enumerate}
\item $R$ is a division ring with $[D:R]_l=[D:R]_r$ is finite (of minimum degree).
\item $R$ is an Ore $G$-domain with division ring of quotients $D$.
\item $N(R)=U(R)\cup \{0\}$, where $N(R)$ is the normalizer of $R$ in $D$.
\end{enumerate}
In particular, if $F\subsetneq C_D(R)$, the centralizer of $R$ in $D$, then $R=C_D(\beta)$ is a division ring, for each $\beta\in C_R(R)\setminus F$, $[D:R]_l$ is finite if and only if $\beta$ is algebraic over $F$, $[D:R]_l=[D:R]_r=[F[\beta]:F]$ and $C_R(R)=F[\beta]$. On the other hand if $R$ does not contain $F$, then $R\cap F=C_R(R)$ is a maximal subring of $F$. Consequently, if a division ring $D$ has a noncentral element which is algebraic over the center of $D$, then $D$ has a maximal subring. In particular, we prove that  if $D$ is a non-commutative division ring with center $F$, then either $D$ has a maximal subring or $dim_F(D)\geq |F|$. In Section 3, we study certain maximal subrings of division rings. In fact we want to show that certain maximal subrings of division rings are valuation rings, similar to non-field maximal subrings of a field. We prove that if a division ring $D$ has a maximal subring $R$ which is a left duo ring, then either $R$ is a division ring or $R$ is an Ore $G$-domain with exactly one nonzero prime ideal $M=R\setminus U(R)$; in fact $R$ is a duo ring and if $R$ is not a division ring, then $R$ is a valuation for $D$ (i.e., for each $x\in D$ either $x\in R$ or $x^{-1}\in R$), $Max_l(R)=Max_r(R)=Max(R)=Spec(R)\setminus\{0\}=\{M\}$, and for each $d\in D$, we have $dRd^{-1}=R$ and $dMd^{-1}=M$. We see that if $D$ is a division ring and $R$ is a certain maximal subring (of the form $({\bf 1-CNO})$ or $({\bf 2-CEO})$ of Theorem \ref{t6}) and $\beta\in R$ such that $\beta R=R\beta\neq R$, then $\bigcap_{n=1}^\infty R\beta^n=0$ and $R$ is a (left and right) bounded ring. In particular, if $R$ is a left duo ring and $R\beta$ is a prime ideal of $R$, then $R$ is a discrete valuation ring for $D$. We show that if $D$ is a division ring and $R$ is a maximal subring of $D$ which is not a division ring, then $R$ is an abelian valuation for $D$ if and only $D^c\subseteq R$ and in this case, $R$ is a duo ring. Finally, we prove that if $D$ is an existentially complete division ring over a field $K$, then $D$ has a maximal subring of the form $C_D(x)$ where $D$ is finite over it. Moreover, if $R$ is a maximal subring of $D$ with $K\subsetneq C_R(R)$, then $R=C_D(x)$ for some $x\in D\setminus K$, which is algebraic over $K$.\\

All rings in this paper are associative unital rings with $1\neq 0$ and all subrings, ring homomorphisms and modules are unital too. A proper subring $R$ of a ring $T$ is called a maximal subring if $R$ is maximal respect to inclusion in the set of all proper subrings of $T$. In fact, a proper subring $R$ of a ring $T$ is a maximal subring of $T$ if and only if for each $\alpha\in T\setminus R$ we have $R[\alpha]=T$. If $R$ is a maximal subring of a ring $T$, then the ring extension $R\subseteq T$ is called a minimal ring extension too. We refer the reader to \cite{adjex,dbsid,dbsc,frd,abmin} for the minimal extension of commutative rings and \cite{dorsy} for the minimal ring extension of non-commutative rings too. Also, we refer the reader to \cite{azn,azffs,azconch,azkra,azkrf,azkrc,azkrh,modica} for maximal subrings of commutative rings and \cite{azcond,azid,klein,laffey,lee} for maximal subrings of non-commutative rings. If $T$ is a ring and $X\subseteq T$, then the centralizer of $X$ in $T$ is denoted by $C_T(X)$ (for $X=\{x\}$, we use $C_T(x)$). In particular, $C(T):=C_T(T)$ is the center of $T$. If $K$ is a field which is contained, as subring, in the center of the ring $T$, then the vector dimension of $T$ over $K$ is denoted by $dim_K(T)$. If $T$ is a ring and $D$ is a subring of $T$ which is a division ring, then $[T:D]_l$ and $[T:D]_r$ denote the vector dimensions of $T$ over $D$ as left and right vector spaces, respectively. If $R$ is a ring and $M$ is a left (resp. right) $R$-module, then the uniform dimension of $M$ is denoted by $u.dim({}_RM)$ (resp. $u.dim(M_R)$). If $D$ is a division ring and $R$ is a subring of $D$, then we define $N_l(R):=\{x\in D\ |\ Rx\subseteq xR\}$, $N_r(R):=\{x\in D\ |\ xR\subseteq Rx\}$ and $N(R):=\{x\in D\ |\ xR=Rx\}$. It is clear that $N(R)=\{0\}\cup\{x\in D^*\ |\ xRx^{-1}=R\}$, hence is called the normalizer of $R$ in $D$. If $T$ is a ring, $Max_r(T)$, $Max_l(T)$, $Max(T)$ and $Spec(T)$, denote the set of all maximal right ideals of $T$, the set of all maximal left ideals of $T$, the set of all maximal ideals of $T$ and the set of all prime ideals of $T$, respectively. For a ring $T$, $T^*=T\setminus \{0\}$ and $U(T)$ is the set of all units of $T$. For the definitions of right/left Ore sets, left/right bounded rings, $G$-rings, valuations, left duo rings, discrete valuation rings, abelian valuation rings and existentially complete division rings we refer the reader to standard text books \cite{cohndr,good,hazw,hfad,lam,lam2,marbo,rvn}.

\section{Maximal subrings in Division Rings}
In this section we study the properties of maximal subrings in division rings. As we mentioned in the introduction of this paper, in fact we want to show that maximal subrings of division rings have properties similar to maximal subrings of fields. We begin this section by some simple facts about maximal subrings in division rings as follows.

\begin{lem}\label{t1}
Let $D$ be a division ring and $R$ be a maximal subring of $D$. If $\lambda\in D$, then $\lambda R\subseteq R\lambda\Longleftrightarrow R\lambda\subseteq \lambda R \Longleftrightarrow \lambda R=R\lambda$.
\end{lem}
\begin{proof}
We may assume that $\lambda\neq 0$. First note that for each nonzero $\lambda\in D$, it is easy to see that $\lambda R\lambda^{-1}$ is a maximal subring of $D$. Thus $\lambda R\subseteq R\lambda$ if and only if $\lambda R\lambda^{-1}\subseteq R$ if and only if $\lambda R\lambda^{-1}=R$ (by the maximality of $\lambda R\lambda^{-1}$) if and only if $\lambda R=R\lambda$. By a similar argument, $R\lambda\subseteq \lambda R$ if and only if $\lambda R=R\lambda$.
\end{proof}

As an immediate consequence of the above lemma is that if $R$ is a maximal subring of a division ring $D$, then $R$ is a left duo ring if and only if $R$ is a right duo (if and only if $R$ is a duo ring).

\begin{lem}\label{t2}
Let $D$ be a division ring and $R$ be a maximal subring of $D$. Then either there exists $\lambda \in D\setminus R$ such that $\lambda R\subseteq R\lambda$ ($R\lambda\subseteq \lambda R$) or $N_l(R)=N_r(R)=N(R)=U(R)\cup\{0\}$.
\end{lem}
\begin{proof}
By the previous lemma, it is clear that $N:=N_l(R)=N_r(R)=N(R)$. Also it is obvious that $U(R)\cup\{0\}\subseteq N$. Now, we have two cases: $(1)$ $N\nsubseteq R$, which means there exists $\lambda\in D\setminus R$ such that $\lambda R\subseteq R\lambda$ and we are done; $(2)$ $N\subseteq R$, now for each $0\neq \lambda\in N$, clearly $\lambda^{-1}\in N$ and therefore $\lambda\in U(R)$. Hence we are done.
\end{proof}

We need the following remark for our next results.

\begin{rem}\label{t3}
Let $D$ be a division ring and $R$ be a subring of $D$. Assume that there exist $\lambda_i\in D$, $1\leq i\leq n$ such that $R\lambda_i=\lambda_i R$ for each $i$, and $D=R\lambda_1+R\lambda_2+\cdots+R\lambda_n$ (i.e., $D$ is a finite normalizing extension of $R$). Then $R$ is a division ring, see \cite[(i) of Corollary P. 369]{macrob}. In particular, if there exists $\lambda\in D\setminus R$, such that $R\lambda=\lambda R$ and $D=R+R\lambda+\cdots+R\lambda^n$ for some natural number $n$, then $R$ is a division ring.
\end{rem}

Now we are ready to present our first main result which shows that a non-division maximal subring of a division ring has properties similar to non-field maximal subrings of fields.

\begin{thm}\label{t4}
Let $D$ be a non-commutative division ring and $R$ be a maximal subring of $D$. Assume that there exists a nonzero non-unit $a$ in $R$, such $aR=Ra$. Then the following holds:
\begin{enumerate}
\item $X=\{1,a, a^2,\ldots\}$ is an Ore set and $D=X^{-1}R=RX^{-1}$.
\item each nonzero one sided ideal of $R$ contains a power of $a$. In particular, $u.dim({}_RR)=1$, $u.dim(R_R)=1$ and $R$ is an Ore domain which division ring of quotients is $D$.
\item each nonzero prime ideal of $R$ contains $a$.
\item $a\in J(R)$.
\item $R$ is a non-simple $G$-domain.
\item for each $\lambda\in N(R)$ either $\lambda\in R$ or $\lambda^{-1}\in R$.
\end{enumerate}
\end{thm}
\begin{proof}
First note that by our assumption $a^{-1}\notin R$. We claim that $X=\{1,a,a^{2},\ldots\}$ is a left/right Ore set in $R$, $D=X^{-1}R=\{a^{-n}r\ |\ n\geq 0, r\in R\}$ and $D=RX^{-1}=\{ra^{-n}\ |\ n\geq 0, r\in R\}$. Since $R$ is a maximal subring of $D$ and $a^{-1}\notin R$ we deduce that $D=R[a^{-1}]$. Now note that clearly $a^nR=Ra^n$ for each integer $n$. Therefore we have $D=R+Ra^{-1}+Ra^{-2}+\cdots=R+a^{-1}R+a^{-2}R+\cdots$. These equalities immediately imply that $D=\{a^{-n}r\ |\ n\geq 0, r\in R\}=\{ra^{-n}\ |\ n\geq 0, r\in R\}$. Now for $(1)$, let $r\in R$ and $n\geq 0$, to show that $X$ is a right Ore set we must prove that $rX\cap a^nR$ is a nonempty set. Since $ra^n\in Ra^n=a^nR$, we conclude that there exists $r'\in R$ such that $ra^n=a^nr'\in rX\cap a^nR$. Thus $X$ is a right Ore set and similarly is a left Ore set too. Thus the first part of $(1)$ holds. The second part of $(1)$, is clear by the first part of the proof. For $(2)$, if $I$ is a nonzero left ideal of $R$ and $0\neq x\in I$, then $x^{-1}\in D$. Therefore there exists $n\geq 0$ and $r\in R$ such that $x^{-1}=a^{-n}r$ and thus $a^{n}=rx\in I$. This immediately implies that $u.dim({}_RR)=1$ and similarly $u.dim(R_R)=1$. Thus $R$ is an Ore domain which division ring of quotients is $D$. Thus $(2)$ holds. For $(3)$, let $P$ be a nonzero prime ideal of $R$, and $n$ be a natural number (by $(2)$) such that $a^n\in P$. Since $aRa^{n-1}=Ra^n\subseteq P$, by induction we conclude that $a\in P$, for $P$ is a prime ideal. Thus $(3)$ holds. To see $(4)$, let $M$ be an arbitrary maximal left ideal of $R$. The $M\neq 0$, for $R$ is not a division ring. Thus by $(2)$, there exists $n\geq 1$ such that $a^n\in M$. We may assume that $n$ is the smallest natural number with this property. If $a\notin M$, then $M+Ra=R$ and therefore $Ra^{n-1}=a^{n-1}R=a^{n-1}M+a^{n-1}Ra=a^{n-1}M+Ra^n\subseteq M$, which is absurd by the minimality of $n$. Thus $a\in J(R)$. Hence $(4)$ holds. $(5)$ is an immediate consequence of $(3)$ and $(4)$. Finally for $(6)$, assume that $\lambda\in N(R)$. We may assume that $\lambda\neq 0$. Assume that $\lambda\notin R$, we claim that $\lambda^{-1}\in R$. For otherwise $\lambda\in R[\lambda^{-1}]=D$ and therefore $\lambda$ is left integral over $R$ of degree $k$. Since $\lambda^nR=R\lambda^n$, we immediately conclude that $D=R+R\lambda+\cdots+R\lambda^{k-1}$. Therefore by Remark \ref{t3}, $R$ is a division ring which is absurd. Thus $\lambda^{-1}\in R$.
\end{proof}

Next we have the following main result for arbitrary maximal subrings of division rings.

\begin{thm}\label{t5}
Let $D$ be a non-commutative division ring and $R$ be a maximal subring of $D$. Then at least one of the following holds:
\begin{enumerate}
\item $N(R)=U(R)\cup\{0\}$.
\item $R$ is a division ring and $D$ is finite over $R$ as left and right vector space with a same degree.
\item $R$ is not division ring which satisfies in Theorem \ref{t4}.
\end{enumerate}
\end{thm}
\begin{proof}
Assume that $(1)$ does not hold. Thus by Lemma \ref{t2}, there exists $\lambda\in D\setminus R$ such that $\lambda R\subseteq R\lambda$. Now we have two cases, either $R$ is a division ring or not. First assume that $R$ is a division ring, hence we infer that $\lambda^{-1}\notin R$. Also note that by Lemma \ref{t1}, we conclude that $\lambda R=R\lambda$ and $\lambda^{-1}R=R\lambda^{-1}$. Since $R$ is a maximal subring of $D$ which does not contain $\lambda$ and $\lambda^{-1}$, we infer that $D=R[\lambda]$ and $D=R[\lambda^{-1}]$.  Thus $\lambda\in D=R[\lambda^{-1}]=R+R\lambda^{-1}+R\lambda^{-2}+\cdots$ and $\lambda^{-1}\in D=R[\lambda]=R+R\lambda+R\lambda^2+\cdots$. Therefore $\lambda$ and $\lambda^{-1}$ are left integral over $R$ (and it is not hard to see that $\lambda$ and $\lambda^{-1}$ have a same degree over $R$, say $k$). Thus $[D:R]_l=k$. Also note that since $R\lambda^n=\lambda^nR$ for each integer $n$, we immediately conclude that $[D:R]_r=k$, and therefore $(2)$ holds in this case. Now suppose that $R$ is not a division ring. Then by a similar proof of $(6)$ of Theorem \ref{t4}, we conclude that $a:=\lambda^{-1}\in R$ and therefore $(3)$ holds.
\end{proof}

In the next result we study the behavior of a maximal subring $R$ of a division ring $D$ respect to the center of $D$, i.e., $C_D(D)$, the center of $R$, i.e., $C_R(R)$ and the centralizer of $R$ in $D$, i.e., $C_D(R)$.

\begin{thm}\label{t6}
Let $D$ be a non-commutative division ring with $C_D(D)=F$ and $R$ be a maximal subring of $D$. Then $R$ satisfies in one of the following conditions:
\begin{enumerate}
\vspace{5mm}
\item $F\nsubseteq R$. In this case $R$ is a non-commutative ring, $C_D(R)=F$ and $C_R(R)=R\cap F$. Moreover $R$ satisfies in exactly one of the following:
\vspace{5mm}

  \begin{itemize}
  \item[$({\bf 1-CND})$] $R$ is a division ring and $[D:R]_l=[D:R]_r$ is finite. For each $X\subseteq D$, $R\neq C_D(X)$.
  \item[$({\bf 1-CNO})$] $R$ is an Ore $G$-domain with division ring of quotients $D$ that satisfies in Theorem \ref{t4}.
  \end{itemize}

\vspace{5mm}

\item $F\subseteq R$. In this case $C_R(R)=C_D(R)$ is a field and one of the following holds:
\vspace{5mm}

  \begin{itemize}
  \item[$({\bf 2-CE})$] $C_R(R)=C_D(R)=F$. In this case either $N(R)=U(R)\cup\{0\}$ or one of the following holds:
\vspace{5mm}

     \begin{itemize}
     \item[$({\bf 2-CED})$] $R$ is a division ring. $[D:R]_l=[D:R]_r$ is finite. For each $X\subseteq D$, $R\neq C_D(X)$.
     \item[$({\bf 2-CEO})$] $R$ is an Ore $G$-domain with division ring of quotients $D$ that satisfies in Theorem \ref{t4}..
     \end{itemize}
\vspace{5mm}

  \item[$({\bf 2-CN})$] $F\subsetneq C_R(R)=C_D(R)$. In this case if $\beta\in C_R(R)\setminus F$, then $R=C_D(\beta)$ is a division ring and exactly one of the following holds:
\vspace{5mm}

     \begin{itemize}
     \item[$({\bf 2-CNF})$] $[D:R]_l$ is finite. In fact, $[D:R]_l$ is finite if and only if $\beta$ is algebraic over $F$. In this case, $[D:R]_l=[D:R]_r=[F[\beta]:F]$ and $C_R(R)=F[\beta]$.
     \item[$({\bf 2-CNI})$] $[D:R]_l$ and $[D:R]_r$ are infinite and $N(R)=R$. Moreover, in this case $C_R(R)/F$ is a purely transcendental extension of fields.
     \end{itemize}

  \end{itemize}

\end{enumerate}

\end{thm}
\begin{proof}
First assume that $(1)$ $F\nsubseteq R$. Let $\alpha\in F\setminus R$. Thus $D=R[\alpha]$, for $R$ is a maximal subring of $D$. Since $\alpha\in F=C_D(D)$, we conclude that $D=R+R\alpha+R\alpha^2+\cdots$. We claim that $C_D(R)=F$. Let $\beta\in C_D(R)$ and $x\in D$. Thus there exist a natural number $n$ and $r_i\in R$, $1\leq i\leq n$, such that $x=r_0+r_1\alpha+\cdots+r_n\alpha^n$. Since $\alpha\in F=C_D(D)$ and $\beta\in C_D(R)$ (therefore $r_i\beta=\beta r_i$), we immediately conclude that $x\beta=\beta x$ and thus $\beta\in C_D(D)=F$. Hence $C_D(R)=F$, which immediately shows that $C_R(R)=F\cap R$. It is clear that $R$ is non-commutative in this case. By Theorem \ref{t4}, it is obvious that ${\bf 1-CND}$ or ${\bf 1-CNO}$ holds. Also note that in case ${\bf 1-CND}$, if $R=C_D(X)$ for some $X\subseteq D$, then since $R$ is a proper subring of $D$, we deduce that there exists $x\in X\setminus F$ and therefore $R=C_D(x)$, by maximality of $R$. Hence $x\in C_D(R)=F$ which is absurd.\\

Now assume that $(2)$ $F\subseteq R$. First we show that $C_R(R)=C_D(R)$. Let $x\in C_D(R)\setminus R$. Thus $x\notin F=C_D(D)$, for $F\subseteq R$. Therefore $C_D(x)$ is a proper subring of $D$ which contains $R$. Hence by maximality of $R$, we infer that $R=C_D(x)$ and therefore $x\in R$, which is absurd. Thus $C_D(R)\subseteq R$ which immediately implies that $C_R(R)=C_D(R)$. From the latter equality we conclude that $C_R(R)$ is a field (note that $C_D(R)$ is a division ring and $C_R(R)$ is the center of $R$). Now we have two cases:
\begin{itemize}
\item[$({\bf 2-CE})$] $F=C_R(R)=C_D(R)$. Hence we are done by Theorem \ref{t5}, we deduce that either $N(R)=U(R)\cup\{0\}$ or $({\bf 2-CED})$ and or $({\bf 2-CEO})$ holds.
\item[$(\bf 2-CN)$] Now assume that $C_D(D)=F\subsetneq E=C_R(R)=C_D(R)$. If $\beta\in E\setminus F$, then by maximality of $R$ we immediately conclude that $R=C_D(\beta)$ and therefore $R$ is a division ring in this case. Hence we have one of the following subcases:
    \begin{itemize}
    \item[$({\bf 2-CNF})$] $D$ is finite dimensional as left/right over $R$ which occur if and only if $\beta$ is algebraic over $F$. Moreover in this case $[D:R]_l=[D:R]_r=[F[\beta]:F]$ and $E=F[\beta]$. To see this note that by Double Centralizer Theorem, see \cite[Theorem 15.4]{lam} or \cite[Corollary 3.3.9]{cohndr}, since $R=C_D(\beta)=C_D(F[\beta])$, we have $[D:R]_l=[F[\beta]:F]$ whenever each side is finite and in this case $E=C_D(R)=F[\beta]$. Similarly, $[D:R]_r=[F[\beta]:F]$ and vice versa, hence we are done.

    \item[$({\bf 2-CNI})$] Thus we may assume that $D$ is not finite over $R$ as left/right vector space. In this case $N(R)=R$. For otherwise if $\lambda\in N(R)\setminus R$, then $\lambda^{-1}\notin R$, for $R$ is a division ring. Hence by a similar proof of Theorem \ref{t5}, we conclude that $D$ is finite over $R$ which is absurd. It is clear that in this case $E/F$ is a purely transcendental extension of fields.
    \end{itemize}
\end{itemize}
\end{proof}

In the next result we prove that the contraction of a certain maximal subring of a division ring $D$ to the center of $D$, remains a maximal subring of the center of $D$.

\begin{thm}\label{t7}
Let $D$ be a non-commutative division ring with $C_D(D)=F$ and $R$ be a maximal subring of $D$ which does not contain $F$. Then $C_R(R)=R\cap F$ is a maximal subring of $F$. Moreover, one of the following holds:
\begin{enumerate}
	\item If $R$ is not a division ring, then $R\cap F$ is a one-dimensional valuation ring (which is a $G$-domain with quotient field $F$).
	\item If $R$ is a division ring, then $F\cap R\subset F$ is a finite minimal ring extension of fields (of minimum degree).
\end{enumerate}
\end{thm}
\begin{proof}
We must show that for each $\alpha\in F\setminus R$, we have $(F\cap R)[\alpha]=F$. We have two cases: either $R$ is a division ring or not. First assume that $R$ is not a division ring. It is clear that $\alpha\in N(R)$, and therefore by $(6)$ of Theorem \ref{t4}, $\beta:=\alpha^{-1}\in R$. Thus $\beta\in F\cap R$ and $D=\{r\beta^{-n}\ |\ r\in R, n\geq 0\}$, by $(1)$ of Theorem \ref{t4}. Now, let $x\in F$, thus there exist $r\in R$ and $n\geq 0$ such that $x=r\beta^{-n}$. Hence $r=x\beta^n\in F\cap R$ (note, $\beta=\alpha^{-1}\in F$) and therefore $x=r\beta^{-n}=r\alpha^n\in (F\cap R)[\alpha]$. Thus $(F\cap R)[\alpha]=F$ and therefore $F\cap R$ is a maximal subring of $F$. Since $\beta\in F\cap R$ and $\beta^{-1}=\alpha\notin F\cap R$, we deduce that $F\cap R$ is not a field and therefore $F\cap R$ is a one-dimensional valuation $G$-domain with quotient field $F$ (note that a non-field maximal subring of a field is a one-dimensional valuation domain and therefore is a $G$-domain).\\
Now assume that $R$ is a division ring. Since $R$ is a division ring we conclude that $\alpha^{-1}\notin R$. By a similar proof of $(6)$ of Theorem \ref{t4}, $\alpha$ is integral over $R$ of degree $n=[D:R]_l$, and we have $D=R+R\alpha+\cdots+R\alpha^{n-1}$. Let $x\in F$, thus there exist $r_0,\ldots,r_{n-1}\in R$ such that $x=r_0+r_1\alpha+\cdots+r_{n-1}\alpha^{n-1}$. Now we claim that for each $i$, $r_i\in F=C_D(R)$ (note that as we see in Case $(1)$ of the previous theorem $C_D(R)=F$). Let $z\in R$, then clearly $xz=zx$, now since $\alpha$ is integral over $R$ of degree $n$, we immediately deduce that for each $i$, $r_iz=zr_i$. Hence $r_i\in C_D(R)=F$, i.e., $r_i\in F\cap R$. Therefore $x=r_0+r_1\alpha+\cdots+r_{n-1}\alpha^{n-1}\in (F\cap R)[\alpha]$. Thus $F\cap R$ is a maximal subring of $F$. It is clear that $F\cap R$ is a field and therefore $F\cap R\subset F$ is a minimal ring extension of fields.
\end{proof}

Now we have the following immediate result.

\begin{cor}\label{t8}
Let $F$ be a field without maximal subring and $D$ be a non-commutative division ring with $C(D)=F$. If $R$ is a maximal subring of $D$, then $F\subseteq R$.
\end{cor}

The part $(1)$ of the following remark is about the extension of an automorphism of a maximal subring of a division ring $D$ to an inner automorphism of $D$ similar to the Skolem-Noether Theorem, see \cite[Corollary 3.3.6]{cohndr}.

\begin{rem}\label{t9}
\begin{enumerate}
\item Let $R$ be a maximal subring of a division ring $D$ and $N(R)=U(R)\cup\{0\}$. Then an automorphism $\sigma$ of $R$ can extend to an inner automorphism of $D$ if and only if $\sigma$ is an inner automorphism of $R$. To see this, assume that $\bar{\sigma}$ be the extension of $\sigma$ to an inner automorphism of $D$ and $\lambda\in D^*$ be such that $\bar{\sigma}(x)=\lambda x\lambda^{-1}$, for each $x\in D$. Then for each $x\in R$, we have $\lambda x\lambda^{-1}\in R$, i.e., $\lambda R\lambda^{-1}\subseteq R$. Therefore $\lambda\in N(R)=U(R)\cup\{0\}$ and hence we are done. In particular, if $D$ is an existentially complete division ring and $R$ is a finitely generated subdivision of $D$ which is a maximal subring of $D$, then $N(R)=R$ implies that each automorphism of $R$ is an inner automorphism, by the first part and \cite[Lemma 14.2]{hfad}; (it is interesting to note that if $D$ is a division ring with center $F$ and $R$ is a maximal subring of $D$ which is a division ring, $F\subseteq R$, $N(R)=R$ and $[R:F]$ is finite, then each automorphism of $R$ is inner by Skolem-Noether Theorem, see \cite[Corollary 3.3.6]{cohndr}).
\item Let $D$ be a division ring and $R$ be a maximal subring of $D$ of the form $R=C_D(x)$. Then for each $g\in D^*$, $C_D(gxg^{-1})$ is a maximal subring of $D$. To see this first note that as we mentioned in the proof of Lemma \ref{t1}, $gRg^{-1}$ is a maximal subring of $D$. It is not hard to see that $C_D(gxg^{-1})=gC_D(x)g^{-1}$ and therefore we are done.
\item Let $D$ be a division ring with a finite maximal subring, then $D$ is a finite field. This fact proved by \cite{klein, laffey} for non-commutative rings. Also note that as an immediate consequence of Theorem \ref{t5} is that if $R$ is a finite maximal subring of a division ring $D$, then either $N(R)=R$ or $D$ is finite field.
\end{enumerate}
\end{rem}

Finally in this section we prove some corollaries about the existence of maximal subrings in division ring.

\begin{cor}\label{t10}
Let $D$ be a division ring with center $F$. Then either $D$ has a maximal subring or each $\alpha\in D\setminus F$ is not algebraic over $F$.
\end{cor}
\begin{proof}
Let $\beta\in D\setminus F$ is algebraic over $F$. Hence $[F[\beta]:F]$ is finite and therefore by Double Centralizer Theorem, see \cite[Theorem 15.4]{lam} or \cite[Corollary 3.3.9]{cohndr}, $[D:C_D(\beta)]_l$ is finite too. Now since $C_D(\beta)$ is a proper subring of $D$ and $D$ is finite over it, we conclude that $D$ has a maximal subring (which contains $C_D(\beta)$, say $R$, and $[D:R]_l$ is finite too. Moreover $[D:R]_l=[D:R]_r$).
\end{proof}

\begin{cor}
	Let $D$ be a division ring with center $F$ which is algebraic over its center. Then $D$ has no maximal subrings if and only if $D=F$ is a field without maximal subrings.
\end{cor}

\begin{cor}\label{t11}
Let $D$ be a non-commutative division ring with center $F$. Then either $D$ has a maximal subring or $dim_F(D)\geq |F|$.
\end{cor}
\begin{proof}
Assume that $dim_F(D)<|F|$. Since $D$ is a non-commutative ring we conclude that $F$ is infinite. Now similar to the proof of \cite[Theorem 4.20]{lam}, we conclude that each $r\in D\setminus F$ is algebraic over $F$ and hence we are done by Corollary \ref{t10}.
\end{proof}

For the proof of the previous corollary note that in fact by the proof of \cite[Theorem 4.20]{lam}, if $D$ is a division ring with center $F$, $a\in D\setminus F$ and $[F(a):F]<|F|$, then $a$ is algebraic over $F$.

\section{Certain Maximal Subrings of Division Rings}
In this section we interested to show that certain maximal subrings of division rings are valuation like non-field maximal subrings of field. First we give a result similar to the extension of a valuation from the center $F$ of a division ring $D$ to a valuation for $D$. Before present the next result let us remind the reader a fact about valuation rings in division rings. Let $D$ be a division ring which is finite over its center $F$ and $V$ be a valuation for $F$. Then $D$ has a total valuation ring $W$ such that $W\cap F=V$ if and only if the set $T=\{\alpha\in D\ |\ \alpha\ \text{is integral over V}\ \}$ is a subring of $D$, see \cite[Theorem 8.12]{marbo}. Now we have the following result without assuming $D$ is finite over its center for extension of a maximal subring $V$ from the center $F$ of a division ring $D$ to a maximal subring $W$ of $D$ with $W\cap F=V$.

\begin{prop}\label{t12}
Let $D$ be a division ring which is algebraic over its center, say $F$. Assume that $V$ is a maximal subring of $F$ which is not a field and $x\in F\setminus V$. If $T=\{\alpha\in D\ |\ \alpha\ \text{is integral over V}\ \}$ is a subring of $D$, then $D$ has a maximal subring $W$ such that $T\subseteq W$, $x\notin T$ and $F\cap W=V$.
\end{prop}
\begin{proof}
Since $V$ is not a field, we infer that $V$ is a one-dimensional valuation domain. Therefore $\alpha=x^{-1}\in V$. Hence $F=V[\alpha^{-1}]$, for $V$ is a maximal subring of $F$. Therefore each element of $F$ is of the form $v\alpha^{-k}$ for some $v\in V$ and $k\geq 0$. Now by definition of $T$ we have $\alpha^{-1}\notin T$, for $V$ is integrally closed in $F$ and $\alpha^{-1}\in F\setminus V$. We prove that $T[\alpha^{-1}]=D$. Assume that $y\in D$, since $y$ is algebraic over $F$, we conclude that there exist $n$ and $a_0,\ldots, a_{n-1}$ in $F$ such that $y^{n}+a_{n-1}y^{n-1}+\cdots+a_1y+a_0=0$. Thus there exist $v_i\in V$ and $m\geq 0$ such that $a_i=v_i\alpha^{-m}$. Therefore $y^{n}+v_{n-1}\alpha^{-m}y^{n-1}+\cdots+v_1\alpha^{-m}y+v_0\alpha^{-m}=0$. Multiplying this equation by $\alpha^{mn}$ (note $\alpha\in F$), we obtain that $(\alpha^{m}y)^n+v_{n-1}(\alpha^m y)^{n-1}+v_{n-2}\alpha^{m}(\alpha^my)^{n-2}+\cdots+v_1\alpha^{m(n-2)}(\alpha^my)+v_0\alpha^{m(n-1)}=0$. Therefore $\alpha^my\in T$. Hence $y\in T[\alpha^{-1}]$, which shows that $T[\alpha^{-1}]=D$. Thus by a natural use of Zorn's Lemma, $T$ has a maximal subring $W$ such that $T\subseteq W$ and $\alpha^{-1}\notin W$. Finally, note that $W\cap F$ is a subring of $F$ which contains $V$ but not $\alpha^{-1}$. Thus $W\cap F=V$, by maximality of $V$ and hence we are done.
\end{proof}

In the next result we prove that if a maximal subring of a division ring is a left duo ring, then it has a similar properties in commutative case (i.e., properties of maximal subrings in fields).

\begin{prop}\label{t13}
Let $D$ be a division ring with a maximal subring $R$. If $R$ is a left duo ring, then either $R$ is a division or $R$ is an Ore $G$-domain with exactly one nonzero prime ideal $M=R\setminus U(R)$. Moreover, in fact $R$ is a duo ring, and if $R$ is not a division ring, then $R$ is a valuation for $D$ (i.e., for each $x\in D$ either $x\in R$ or $x^{-1}\in R$), $Max_l(R)=Max_r(R)=Max(R)=Spec(R)\setminus\{0\}=\{M\}$, and for each $d\in D^*$, we have $dRd^{-1}=R$ and $dMd^{-1}=M$.
\end{prop}
\begin{proof}
If $R$ is a division ring we are done. Hence assume that $0\neq x\in R\setminus U(R)$. Thus $R[x^{-1}]=D$ by maximality of $R$. Since $R$ is a left duo ring, we conclude that $xR\subseteq Rx$ and therefore $xR=Rx$ (hence $R$ is a duo ring), by Lemma \ref{t1}. Thus $x^{-1}R=Rx^{-1}$ which implies that $D=R[x^{-1}]=R+Rx^{-1}+Rx^{-2}+\cdots$. Hence each element of $D$ is of the form $rx^{-n}$ for some $r\in R$ and $n\geq 0$ (note $x\in R$ and see the proof of Theorem \ref{t4}). It is clear that by Theorem \ref{t4}, $R$ is an Ore $G$-domain with division ring of quotients $D$. Since $0\neq x\in R\setminus U(R)$, we infer that $xR=Rx$ is a nonzero proper ideal of $R$, therefore $R$ has a nonzero prime ideal. Now assume that $M$ is a nonzero prime ideal of $R$ and $0\neq y\in M$. Thus $y^{-1}=rx^{-n}$ for some $r\in R$ and $n\geq 0$. Hence $x^n=yr\in M$. Since $M$ is prime and $R$ is duo we conclude that $x\in M$. Thus $R\setminus U(R)\subseteq M$ and therefore $M=R\setminus U(R)$ is the only nonzero prime ideal of $R$. Now note that since $R$ is a duo ring we infer that for each $a\in R$ we have $aR=Ra$ and therefore $aRa^{-1}=R$. Also by $aRa^{-1}=R$, we immediately deduce that $aMa^{-1}$ is a nonzero prime ideal of $R=aRa^{-1}$ and thus $aMa^{-1}=M$. Now if $d\in D^*$, then $d=rx^{-n}$ for some $r\in R$ and $n\geq 0$. Hence we have $dRd^{-1}=rx^{-n}Rx^nr^{-1}$. Since $x\in R$, we have $x^{-n}Rx^n=R$, and therefore $dRd^{-1}=rRr^{-1}=R$, for $r\in R$. Similarly we have $dMd^{-1}=M$. Finally, we prove that $R$ is a valuation for $D$. Assume that $d\in D$ and $R$ does not contain $d^{-1}$ and $d$. Thus $R[d]=D=R[d^{-1}]$, for $R$ is a maximal subring of $D$. Since $Rd=dR$, similar to the proof of $(6)$ of Theorem \ref{t4}, we conclude that $D=R+Rd+\cdots+Rd^{k}$ for some $k$. Since $Rd^i=d^iR$, we conclude that $R$ is a division ring, by Remark \ref{t3}, which is absurd. Thus $R$ is a valuation for $D$.
\end{proof}

In the next result we see that a maximal subring of a division ring given in Theorem \ref{t4} is a bounded ring; moreover, in a certain case it is a discrete valuation ring too.

\begin{prop}\label{t14}
Let $D$ be a division ring, $R$ be a maximal subring of $D$ and $\beta\in R$ such that $\beta R=R\beta\neq R$. Then $\bigcap_{n=1}^\infty R\beta^n=0$ and $R$ is a (left/right) bounded ring. In particular, if $R$ is a left duo ring and $R\beta$ is a prime ideal of $R$, then $R$ is a discrete valuation ring for $D$.
\end{prop}
\begin{proof}
First note that for each $n\in\mathbb{N}$, $R\beta^n=\beta^nR$ is an ideal of $R$ and by $(2)$ of Theorem \ref{t4}, for each nonzero left/right ideal $I$ of $R$, $I$ contains $\beta^n$ for some natural number $n$, and therefore $I$ contains $R\beta^n=\beta^nR$. Thus $R$ is a left/right bounded ring (note that $u.dim_l(R)=u.dim_r(R)=1$ and therefore the ideal $R\beta^n=\beta^nR$ is essential as a left/right ideal). Hence if $B:=\bigcap_{n=1}^\infty R\beta^n\neq 0$, then there exists a natural number $m$ such that $B=R\beta^m$. Thus $R\beta^m=R\beta^{m+1}$, i.e., $\beta$ is a unit of $R$ which is absurd. Thus $B=0$. Now assume that $R$ is a left duo ring and $M=R\beta$ is a prime ideal of $R$. Thus by Proposition \ref{t13}, we conclude that $R$ is a valuation ring for $D$ and $M$ is the only nonzero prime ideal of $R$, i.e., $U(R)=R\setminus M$. Since $B=0$, we deduce that each element of $R$, is of the form $u\beta^{m}$ for some $u\in U(R)$ and $m\geq 0$ and clearly $m$ is unique. This also shows that each element of $D$ is of the form $u\beta^m$ for some $u\in U(R)$ and a unique element $m\in\mathbb{Z}$. Thus $R$ is a left discrete valuation for $D$ (see \cite[Exersice 19.7]{lam}). Similarly $R$ is a right discrete valuation ring too and hence we are done.
\end{proof}

The following which shows when a maximal subring of a division ring is an abelian valuation is in order.

\begin{prop}\label{t15}
Let $D$ be a division ring and $R$ is a maximal subring of $D$ which is not a division ring. Then $R$ is an abelian valuation for $D$ if and only $D^c\subseteq R$. In this case, $R$ is a duo ring.
\end{prop}
\begin{proof}
This is an immediate consequence of \cite[Lemma 9.2.2]{cohndr}, for $R$ is a maximal subring of $D$. For the final part note that for each $a,b\in R$, $aba^{-1}b^{-1}\in R$. Thus $aba^{-1}\in Rb\subseteq R$ and therefore $ab\in Ra$. Hence $aR\subseteq Ra$ which immediately implies that $aR=Ra$, by Lemma \ref{t1}. Thus $R$ is a duo ring and hence we are done.
\end{proof}

The following result is about the existence and also the form of maximal subrings in an existentially complete division ring.

\begin{prop}\label{t16}
Let $D$ be an existentially complete division ring over a field $K$. Then the following hold:
\begin{enumerate}
\item $D$ has a maximal subring of the form $C_D(x)$ where $D$ is finite over it.
\item If $R$ is a maximal subring of $D$ with $K\subsetneq C_R(R)$, then $R=C_D(x)$ for some $x\in D\setminus K$, which is algebraic over $K$.
\end{enumerate}
\end{prop}
\begin{proof}
First note that $C_D(D)=K$ by \cite[Corollary 14.5]{hfad} which is the prime subfield of $D$ (i.e., $K=\mathbb{Q}$ or $\mathbb{Z}_p$ for a prime number $p$) and $D$ contains an algebraically closed field $E$ of infinite transcendental degree over $K$, by \cite[Corollary 14.12]{hfad}. Thus $D$ contains the algebraic closure of $K$, say $L$. Since $C_D(D)=K$ and $K\subsetneq L$ is an algebraic extension, we conclude that $D$ has a noncentral element which is algebraic over its center (i.e., each element of $L\setminus K$). Therefore by Corollary \ref{t10}, $(1)$ holds. Now assume that $D$ has a maximal subring $R$ with $K\subsetneq C_R(R)$ and each element of $C_R(R)\setminus K$ is not algebraic over $F$. Suppose that $\beta\in C_R(R)\setminus K$, then clearly $R=C_D(\beta)$. We claim that $C_R(R)=K(\beta)$. It is clear that $K(\beta)\subseteq C_R(R)$. Conversely, assume that $\alpha\in C_R(R)\setminus K(\beta)$, then $R=C_D(\alpha)$ and therefore $C_D(\beta)=C_D(\alpha)$. Hence by \cite[Lemma 14.3]{hfad}, $K(\alpha)=K(\beta)$, which is a contradiction for $\alpha\in C_R(R)\setminus K(\beta)=C_R(R)\setminus K(\alpha)$, but $\alpha\in K(\alpha)$. Thus $C_R(R)=K(\beta)$. Now since $\beta$ is not algebraic over $K$ we immediately conclude that $K(\beta^2)\subsetneq K(\beta)$, but $K(\beta)=C_R(R)=K(\beta^2)$, for $\beta^2\notin K$. Hence we conclude that $C_R(R)/K$ is an algebraic extension and hence we are done.
\end{proof}

We conclude this paper with the following remark. In the next remark we use the notation of \cite[Sec 3.2]{cohndr}, i.e., the Jacobson-Bourbaki correspondence.

\begin{rem}\label{t17}
\begin{enumerate}
\item Let $D$ be a division ring and $E:=End_{\mathbb{Z}}(D)$. Then by \cite[Theorem 3.2.4]{cohndr} (i.e., the Jacobson-Bourbaki correspondence) there exists  an order-reversing bijection between the subdivisions $D'$ of $D$ and closed $\rho(D)$-subrings $F$ of $E$ (with finite topology). Consequently, $D'$ is a maximal subring of $D$ which is a division ring and $[D:D']_l<\infty$ if and only if $E$ has a closed $\rho(D)$-subring $F$ where $\rho(D)$ is a maximal subring of $F$ and $[F:\rho(D)]_r<\infty$. In fact, if $D'$ is a maximal subring of $D$ which is a division ring and $[D:D']_l<\infty$, then $F=End({}_{D'}D)$ (is a closed finite minimal extension of $\rho(D)$ inside $E$); conversely, if $F$ is a (closed) finite (i.e., $[F:\rho(D)]_r<\infty$) minimal extension of $\rho(D)$ in $E$, then $D':=\{x\in D\ |\ (xy)f=x((y)f),\ y\in D,\ f\in F\}$ (note, writing maps on the right) is a maximal subring of $D$ which $[D:D']_l<\infty$; and in fact $[F:\rho(D)]_r=[D:D']_l$.
\item If $F$ is a $\rho(D)$-subring of $E$ with $[F:\rho(D)]_r<\infty$, then by \cite[Ex.4, P.100]{cohndr} or see \cite{jacobdra}, $F$ is a closed subring in $E$ (by finite topology). In particular, if a division ring $D$ has no maximal subring, then for each noncentral $a\in D$, the map $\lambda_a$ (i.e., $x\longmapsto ax$, for each $x\in D$), is not right algebraic over $\rho(D)$. To see this first note that $\rho(D)$ centralizes $\lambda_a$ in $E$. Hence if $\lambda_a$ is right algebraic over $\rho(D)$, then $F:=\rho(D)[\lambda_a]$ is a $\rho(D)$-subring of $E$ which is finite over $\rho(D)$ (i.e., $[F:\rho(D)]_r<\infty$). Therefore by Jacobson-Bourbaki correspondence, we deduce that $D$ has a maximal subring which is absurd.
\item Let $D$ be a division ring and $a\in D$ be a noncentral element. Assume that $\sigma$ be an inner automorphism of $D$, defined by $\sigma(x)=axa^{-1}$, for each $x\in D$. If $\sigma$ is left algebraic over $\rho(D)$ (in $E$), then $a$ is algebraic over the center of $D$. In particular, $D$ has a maximal subring. To see this, let $n\in\mathbb{N}$ and $b_0,b_1,\ldots,b_{n-1}\in D$ such that
    $$\sigma^n+\rho_{b_{n-1}}\sigma^{n-1}+\cdots+\rho_{b_1}\sigma+\rho_{b_0}=0$$
    note that since $\rho(D)\cong D$ is a division ring, we may assume that $\sigma$ satisfies in a monic polynomial. Moreover, since $\sigma$ is invertible, we may assume that $\rho_{b_0}\neq 0$, i.e., $b_0\neq 0$. Take $x=1$ in the above equation we deduce that $1+b_{n-1}+\cdots+b_0=0$ and if we take $x=a^n$, then we obtain that $a^n+a^nb_{n-1}+\cdots+a^nb_1+b_0=0$. Therefore $a^n(1+b_{n-1}+\cdots+b_1)+b_0=0$, i.e., $a^n(-b_0)+b_0=0$. Thus $a^n=1$, for $b_0\neq 0$. Hence $a$ is algebraic over the prime subring of $D$, in particular $a$ is algebraic over the center of $D$. Therefore $D$ has a maximal subring by Corollary \ref{t10}.
\end{enumerate}
\end{rem}

\centerline{\Large{\bf Acknowledgement}}
The author is grateful to the Research Council of Shahid Chamran University of Ahvaz (Ahvaz-Iran) for
financial support (Grant Number: SCU.MM1403.721)


\end{document}